\documentclass[12pt]{article}
\usepackage{amsmath}
\usepackage{amssymb}
\usepackage{amsfonts}
\usepackage{latexsym}
\usepackage{color}
\usepackage{graphicx}

\catcode `\@=11 \@addtoreset{equation}{section}

\catcode `\@=12

\newcommand{\be}{\begin{equation}}
\newcommand{\en}{\end{equation}}
\newcommand{\bea}{\begin{eqnarray}}
\newcommand{\ena}{\end{eqnarray}}
\newcommand{\beano}{\begin{eqnarray*}}
\newcommand{\enano}{\end{eqnarray*}}
\newcommand{\bee}{\begin{enumerate}}
\newcommand{\ene}{\end{enumerate}}

\newcommand{\Hil}{{\cal H}}

\newcommand{\F}{{\cal F}}

\newcommand{\Sc}{{\cal S}}
\newcommand{\D}{{\cal D}}

\newcommand{\M}{{\cal M}}

\newtheorem{thm}{Theorem}
\newtheorem{example}[thm]{Example}
\newtheorem{cor}[thm]{Corollary}

\newtheorem{prop}[thm]{Proposition}
\newtheorem{rem}[thm]{Remark}

\newenvironment{proof}{\noindent {\bf Proof:}}{\hfill$\Box$}
\begin{document}
\thispagestyle{empty}

\vspace*{1cm}

\begin{center}
{\Large \bf Operators and POVMs generated by Parseval frames} 

\vspace{4mm}
{\large S. Ku\.{z}el${}^{1}$ \quad and \quad P. {\L}ukasiewicz${}^{1}$}   \\
\vspace*{1cm}
\normalsize
\vspace*{.5cm}
${}^1$ AGH University, Krak\'ow, Faculty of Applied Mathematics, Poland.\\

\end{center}

\vspace*{0.5cm}

\begin{abstract}
Let $\F_e$ be a Parseval frame in a Hilbert space 
$\Hil$ and  let 
${\sf E}$ be a set of real numbers. From these data, we construct an operator 
$H_{{\sf E},e}$ and a positive operator-valued measure (POVM) $F_{{\sf E},e}$ This paper investigates in detail the relationship between the operator 
$H_{{\sf E},e}$ and the POVM 
$F_{{\sf E},e}$. Our results extend the classical correspondence between a self-adjoint operator generated by an orthonormal basis and its associated projection-valued (spectral) measure.
\end{abstract}


\vfill
\section{Introduction}\label{s1}

As is well known, in quantum mechanics the dynamics of a closed system $\Sc$ is often determined by a self-adjoint operator, the Hamiltonian $H$ of the system $\Sc$.  A crucial feature of $H$, which plays a significant role in analyzing
$\Sc$, is its spectrum of eigenvalues and the corresponding eigenvectors. If
$H$ has a purely discrete spectrum $\sigma(H)=\{E_j\}_{j\in\mathbb{J}}$, where  $\mathbb{J}\subseteq\mathbb{Z}$ is an index set, then the eigenvectors corresponding to the distinct eigenvalues are mutually orthogonal. In this case, there exists an orthonormal basis $\F_{e^+}=\{e_j^+\}_{j\in\mathbb{J}}$ of the Hilbert space $\Hil^+$
consisting of eigenvectors of $H$. Therefore, any physical state of $\Sc$ can be represented as a linear combination of basic elements $e_j^+$. With respect to this basis, the operator $H$
 admits the spectral representation 
$$
H{\equiv}H_{{\sf E},e^+}=\sum_{j\in\mathbb{J}}{E}_j \langle \cdot,  e_j^+ \rangle{e}_j^+, \qquad {\sf E}=\{E_j\}_{j\in\mathbb{J}}
$$
with domain
$$\mathcal{D}(H_{{\sf E}, e^+})=\{f\in\Hil^+ \ : \ \sum_{j\in\mathbb{J}}{E}_j^2 |\langle f, e_j^+\rangle|^2<\infty\}.
$$

Alternatively, the system $\Sc$ can be described using the projection-valued measure 
$$
F_ {{\sf E}, e^+}(\Delta)=\sum_{j\in\mathbb{J}_\Delta}\langle \cdot, e_j^+\rangle e_j^+, \qquad \mathbb{J}_\Delta=\{j\in\mathbb{J} : E_j\in\Delta\}, \quad \Delta\in\mathcal{B}(\mathbb{R}), 
$$
where $\mathcal{B}(\mathbb{R})$ is the $\sigma$ -algebra of Borel set on $\mathbb{R}$. 

The relationship between $H_{{\sf E},e^+}$ and $F_ {{\sf E}, e^+}$
is remarkably simple and transparent 
\begin{equation}\label{ups11}
H_{{\sf E},e^+}=\int_{\mathbb{R}}\lambda{d} F_ {{\sf E}, e^+}(d\lambda), \qquad \mathcal{D}(H_{{\sf E}, e^+})=\{ f\in\Hil^+ \ : \ \int_{\mathbb{R}}\lambda^2d\langle F_ {{\sf E}, e^+}(d\lambda)f, f \rangle<\infty\}.
\end{equation}

In the analysis of a concrete physical situation, it may occur that \emph{not all} vectors of the  Hilbert space  $\Hil^+$ are  \emph{physically relevant} to describe the system  $\Sc$. 
This happens, for example, when the energy of $\Sc$ cannot realistically become too large, or when $\Sc$ is localized in a bounded spatial region, or when its momentum is constrained within certain limits. In such cases - and others like them - it is natural to introduce a \emph{physical state space}  $\Hil$,  defined as a linear subspace of the \emph{mathematical state space} $\Hil^+$ on which $\Sc$ is originally formulated.  This is exactly the point of view in \cite{jpg}, just to cite one, where $\Hil$ contains the functions of $\Hil^+=L_2(\mathbb{R})$ which are zero outside a certain compact subset $\D$ of $\mathbb{R}$. Depending on the interpretation of $\Hil^+$, this approach can be used to describe particles located in $\D$, or particles with a bounded momentum.

 In general, the physical Hilbert space $\Hil$ can be constructed as a projection of $\Hil^+$, through some suitable orthogonal projector operator $P^+$, i.e., $\Hil=P^+\Hil^+$.  Then, the restriction of $H_{{\sf E},e^+}$ in $\Hil$ leads to the operator 
\begin{equation}\label{uman63r}
H_{{\sf E},e}=P^+H_{{\sf E},e^+}=\sum_{j\in\mathbb{J}}{E}_j \langle \cdot,  e_j \rangle{e}_j, 
\end{equation}
 where $\F_e=\{e_j=P^+e_j^+\}_{j\in\mathbb{J}}$ is a Parseval frame in $\Hil$ (see Sect. \ref{sII} for definition).

 The operator $H_{{\sf E},e}$, as defined in \eqref{uman63r} with an appropriately chosen domain, can be viewed as  \emph{a Hamiltonian  generated by a Parseval frame  $\F_e$ and the set of real numbers ${\sf E}$}. Operators of this type have been studied in \cite{BK1, BK2}, and arise as a particular instance of the broader framework of multipliers \cite{CORSO, BA2}. 
 
For finite Parseval frames, the operators $H_{{\sf E},e}$
 can be interpreted as arising from the Klauder–Berezin–Toeplitz-type quantization \cite{Gaz3, Kla} of a real-valued function $f : \mathcal{X} \to \mathbb{R}$, defined on a discrete set of data 
$\mathcal{X}=\{a_j\}_{j\in\mathbb{J}}$
 associated with a physical system. The eigenvalues of 
$H_{{\sf E}, e}$ then constitute the \emph{quantum spectrum} of the classical observable $f$, while its \emph{classical spectrum} is given by the set of values
${\sf E}=\{E_j=f(a_j)\}_{j\in\mathbb{J}}$
as described in \cite{Cot1, Cot2}. These results can be generalized
to the case of infinite Parseval frames with the use of POVM quantization developed in \cite{Gaz1}.

Similarly to the case of operators  $H_{{\sf E}, e^+} \to H_{{\sf E}, e}$, the restriction of $\Hil^+$ to  $\Hil$ transforms the projection-valued measure $F_{{\sf E}, e^+}$
into a positive operator-valued measure $F_{{\sf E}, e}=P^+F_{{\sf E}, e^+}: \mathcal{B}(\mathbb{R}) \to \mathcal{L}^+(\Hil)$ defined by the formula 
$$
F_ {{\sf E}, e}(\Delta)=\sum_{j\in\mathbb{J}_\Delta}\langle \cdot, e_j\rangle e_j,  \qquad \mathbb{J}_\Delta=\{j\in\mathbb{J} : E_j\in\Delta\}, \quad \Delta\in\mathcal{B}(\mathbb{R}),
$$ 
where $\mathcal{L}^+(\mathcal{H})$ is the set of positive\footnote{An operator $L$ acting in a Hilbert space $\mathcal{H}$ with a scalar product $\langle\cdot, \cdot\rangle$ is called \emph{positive} if $\langle Lf,f\rangle \geq{0}$} bounded operators acting in $\Hil$.

In contrast, to the relation \eqref{ups11} between a self-adjoint operator $H_{{\sf E},e^+}$ and the corresponding projection valued measure $F_ {{\sf E}, e^+}$, the relationship between 
$
F_ {{\sf E}, e}$  and $
H_{{\sf E},e}$
is much more complicated, and its  
investigation 
is \emph{the main objective of this paper}. 

After presenting the basic elements of frame theory and POVM theory in Sect. \ref{sII}, we focus on the POVMs 
$F_ {{\sf E}, e}$, generated by Parseval frames. These POVMs constitute particular cases of framed POVMs  \cite{MHC}. The close relationship between POVMs and frames allows us to draw on the rich existing theory of POVMs to gain new perspectives on frame theory. In particular, this approach yields a new proof of the fundamental inequality for Parseval frames, presented in Sect. \ref{sIII.1}.
   
 In Sect. \ref{sIII.2}, commutative POVMs 
$F_ {{\sf E}, e}$ are completely characterized under the assumption that the set ${\sf E}$
consists of different elements. Theorem \ref{uuu4} shows that frames of this type have a relatively simple structure, based on an orthonormal basis of $\Hil$. This orthonormal basis makes it possible to explicitly construct sharp versions \cite{RB}-\cite{RB2} of a commutative POVM $F_{{\sf E}, e}$ in Corollary \ref{uuu7}. In Sect. \ref{sIII.3}, the joint measurability of two POVMs $F_{{\sf E}, e}$ and $F_{\widehat{\sf E}, \widehat{e}}$ 
is discussed based on the results of
\cite{RB}.

Sect. \ref{sIV} opens with a discussion of how one can construct an operator $H$
that maximally reflects the properties of a POVM $F$. We employ the Naimark method to construct a symmetric operator $H$ associated with a POVM $F$ \cite{AG, Nai}, and we discuss how its domain can be extended in an optimal manner to obtain a self-adjoint operator. This line of argument, for the case where $F=F_{{\sf E}, e}$ leads to an operator 
$H_{{\sf E}, e}$, defined by \eqref{uman63r}, and equipped with one of two possible domains, namely \eqref{busa4} or \eqref{busa4b}.
Theorem \ref{duda1} demonstrates that, for a commutative POVM $F_{{\sf E}, e}$, the associated operator $H_{{\sf E}, e}$ with domain \eqref{busa4b} is a self-adjoint operator.

In Sect. \ref{sIV.2}, we analyze the relationship between the spectrum of POVM $F_{{\sf E}, e}$ 
 and  the spectrum of $H_{{\sf E}, e}$. 
We note that the spectrum of $F_{{\sf E}, e}$
coincides with the closure of 
${\sf E}$ and can be interpreted as the classical spectrum of an observable 
$f$, while the spectrum of 
$H_{{\sf E}, e}$ can be regarded as its quantum spectrum,
 as proposed in \cite{Cot1, Cot2}.
 
 In the case of a commutative POVM based on a Parseval frame 
$\F_e$ with finite excess ${\bf e}[\F_e]=n$, the eigenvalues of 
$H_{{\sf E}, e}$ coincide with 
 ${\sf E}$ up to 
$n$ values. The situation becomes considerably more interesting when 
${\bf e}[\F_e]=\infty$. In this case, the set ${\sf E}$ can be transformed into $\sigma(H_{{\sf E}, e})$
in a wide variety of ways (see \ref{eigen12}), using an appropriate choice of a Parseval frame $\F_e$.
 
In the case of non-commutative POVMs based on Parseval frames that possess additional properties relevant for applications, one may also expect predictable relationships between 
${\sf E}$ and $\sigma(H_{{\sf E}, e})$.
In Sect. \ref{sIV.3}, we confirm this hypothesis for Parseval frames associated with specific classes of optimal Grassmannian frames, including Mercedes-type Parseval frames and frames derived from conference matrices.

\section{Preliminary definitions and results}\label{sII}
In the following, we present the necessary facts about frames and POVMs in a form that is convenient for our exposition. More details can be found in \cite{chri, Heil, MHC, SH}.

\subsection{Frames}\label{sII.1}
 
Let $\mathcal{H}$ be a separable Hilbert space with linear scalar product $\langle \cdot, \cdot \rangle$ in the first argument.   

\emph{A frame}  is a family of vectors 
$\F_e=\{e_j\}_{j\in\mathbb{J}}$  in $\mathcal{H}$  for which 
$$
A||f||^2\leq\sum_{j\in\mathbb{J}}{|\langle f, e_j \rangle|^2}\leq{B}||f||^2, \qquad f\in \mathcal{H}, \quad 0<A\leq{B}.
$$

A frame 
$\F_e$
 is said to be \emph{$A$-tight} if $A=B$, and \emph{normalized} if 
$\|e_j\|=1 $ for $j\in\mathbb{J}$. For each $A$-tight frame $\F_e$,
\begin{equation}\label{auau1}
f=\frac{1}{A}\sum_{j\in\mathbb{J}}\langle f, e_j\rangle{e_j}, \qquad f\in\Hil.
\end{equation}
The coefficients in \eqref{auau1} are not uniquely determined and there exist alternative representations $f=\frac{1}{A}\sum_{j\in\mathbb{J}}{c_j}{e_j}$ where the coefficients $c_j$ differ from those of \eqref{auau1}. The range of available coefficients 
can be evaluated using the concept of frame excess.  We recall that the \emph{excess} ${\bf e}[\F_e]$ of $\F_e$ is the highest integer $n$ such that  $n$ elements  can be deleted from $\F_e$ and still leave a complete set, or $\infty$ if there is no upper bound to the number of elements that can be removed. The zero excess of a $A$-tight frame $\F_{e}$ means that $\F_e$ is a Riesz basis  and the coefficients in \eqref{auau1} are uniquely determined.

An $1$-tight frame $\F_e$ is called \emph{a Parseval frame}.  Every $A$-tight
 frame $\F_{e'}=\{e_j'\}_{j\in\mathbb{J}}$ can be rescaled to obtain a Parseval frame $\F_{e}=\{e_j\}_{j\in\mathbb{J}}$ by setting $e_j' \to e_j=\frac{1}{\sqrt{A}}e_j'$.

The Naimark dilation theorem is stated for Parseval frames in the following form \cite{hanlar}.  
\begin{thm}\label{Naimark}
Let $\F_e=\{e_j\}_{j\in\mathbb{J}}$ be a Parseval frame in a Hilbert space $\mathcal{H}$. 
Then there exist a Hilbert space $\M$ and a Parseval frame
$\F_{m}=\{m_j\}_{j\in\mathbb{J}}$ in $\M$ such that the set of vectors
\begin{equation}\label{ups1}
\F_{e^+}=\{e_j^+=e_j\oplus{m}_j\}_{j\in\mathbb{J}}
\end{equation}
is an orthonormal basis of $\Hil^+=\mathcal{H}\oplus\M$. The excess ${\bf e}[\F_e]$ of $\F_e$ coincides with the dimension of the Hilbert space $\M$, while the excess ${\bf e}[\F_m]$ of $\F_m$ coincides with $\dim\Hil$. 
\end{thm}

It is also straightforward to prove an inverse statement: given an orhonormal basis 
$\F_{e^+}$ of $\Hil^+$ and an orthogonal projection $P^+$ in $\Hil^+$, the set
$P^+\F_{e^+}=\{e_j=P^+e_j^+\}_{j\in\mathbb{J}}$ forms a Parseval frame in $\mathcal{H}=P^+\mathcal{H}^+$. 

A normalized frame $\F_e=\{e_j\}_{j\in\mathbb{J}}$ acting in a finite-dimensional Hilbert space $\Hil$ is called \emph{equiangular} if there exists a constant $\alpha$ such that 
$
\alpha=|\langle e_j, e_i \rangle|, \ j\not=i\in\mathbb{J}.
$
The constant $\alpha$ is called \emph{the coherence constant} and measures how close an equiangular frame $\F_e$ is to being orthogonal.
 
Every equiangular 
$A$-tight frame 
$\F_e$ is referred to as  \emph{an optimal Grassmannian frame}. Optimal Grassmannian frames $\F_e=\{e_j\}_{j\in\mathbb{J}}$ exist only for certain combinations of the parameters  $|\mathbb{J}|$ and ${\bf e}[\F_e]$. 
More precisely, in complex (respectively, real) Hilbert spaces, the existence of such frames requires that $(|\mathbb{J}|-{\bf e}[\F_e])^2\geq|\mathbb{J}|+{\bf e}[\F_e]$ (respectively, $(|\mathbb{J}|-{\bf e}[\F_e])^2\geq|\mathbb{J}|$). In both cases, the coherence constant $\alpha$ coincides with 
$\sqrt{\frac{{\bf e}[\F_e]}{(|\mathbb{J}|-{\bf e}[\F_e])(|\mathbb{J}|-1)}}$, and $\F_e$ is a 
$A$-tight frame with frame bound 
$A=\frac{|\mathbb{J}|}{|\mathbb{J}|-{\bf e}[\F_e]}$.

\subsection{Positive operator valued measure}
 \emph{A normalized positive operator-valued measure} (for short, a POVM) is defined as a map $F : \mathcal{B}(\mathbb{R}) \to \mathcal{L}^+(\Hil)$ such that\footnote{A POVM is typically defined as a $\sigma$-additive map from the Borel $\sigma$-algebra $\mathcal{B}(X)$ of a topological space $X$. In the special case where 
$X=\mathbb{R}$, the term real POVM is used. Throughout the paper, we consider only real POVMs.}
$$
F\left(\bigcup_{n=1}^\infty\Delta_n\right)=\sum_{n=1}^\infty{F(\Delta_n)}, \qquad F(\mathbb{R})=I,
$$
where $\{\Delta_n\}$ is a countable family of disjoint sets in $\mathcal{B}(\mathbb{R})$ and the series on the right-hand side is understood as the strong limit of the corresponding sequence of partial sums.

The spectrum $\sigma(F)$ of a POVM $F$ is the closed set 
\begin{equation}\label{uman63kk}
\{\lambda\in\mathbb{R} \ : \forall\Delta \ \mbox{open}, \ \lambda\in\Delta\ \Rightarrow F(\Delta)\not=0   \}.
\end{equation}

A POVM $F$ is said to be \emph{commutative} if 
$$
[F(\Delta_1), F(\Delta_2)]:=F(\Delta_1)F(\Delta_2)-F(\Delta_2)F(\Delta_1)=0, \qquad \Delta_j\in\mathcal{B}(\mathbb{R})
$$
and \emph{orthogonal} if $\Delta_1\cap\Delta_2=\emptyset$ implies
$F(\Delta_1)F(\Delta_2)=0$.

\emph{A projection valued measure} (for short, a PVM) is an orthogonal POVM or, equivalently, is a POVM $F$ such that $F(\Delta)$ is an orthogonal projection operator for each $\Delta\in\mathcal{B}(\mathbb{R})$. 


\section{POVMs generated by Parseval frames}
\subsection{Definition and elementary properties}\label{sIII.1}
Let $\F_e=\{e_j\}_{j\in\mathbb{J}}$ be a Parseval frame in a Hilbert space $\mathcal{H}$, where   $\mathbb{J}$ is a finite or countable subset of integers, and let ${\sf E}=\{E_j\}_{j\in\mathbb{J}}$ be the set of real numbers.
Using $\F_e$ and ${\sf E}$, one can  define the POVM
\begin{equation}\label{uman63b}
F_ {{\sf E}, e}(\Delta)=\sum_{j\in\mathbb{J}_\Delta}\langle \cdot, e_j\rangle e_j, \qquad \mathbb{J}_\Delta=\{j\in\mathbb{J} : E_j\in\Delta\}, \qquad \Delta\in\mathcal{B}(\mathbb{R}).
\end{equation}

 It is easy to see that multiplying each frame element $e_j$ by a unimodular constant (i.e., a complex number of modulus one) yields another Parseval frame. From the POVMs perspective, this transformation does not alter the measurement statistics. In other words, replacing the frame elements $\{e_j\}$ with $\{e^{i\theta_j}e_j\}$, leaves 
$F_{{\sf E}, e}$  unchanged. Thus, POVM 
$F_{{\sf E}, e}$ is defined by the Parseval frame $\F_e$ up to multiplication by unimodular constants.

In general, a Parseval frame 
$\F_e$ may contain zero vectors 
$e_j=0$. However, in our setting, the presence of such vectors does not affect the definition of POVM \eqref{uman63b}, while it may introduce unnecessary ambiguity regarding the numbers $E_j$ associated with them. For this reason, in what follows we assume that our Parseval frames \emph{contain no zero vectors}, i.e. 
$$
\F_e=\F_{e\not=0}=\{e_j\in\F_e : e_j\not=0\}.
$$

The formula \eqref{uman63b} provides a bridge to the rich theory of POVMs, offering new perspectives for the study of frames. The importance of the relationship between POVM and frame theory is highlighted by the following simple proof of the fundamental identity for Parseval frames, originally established in \cite{BCE} through intrinsic methods of frame theory. Specifically, one obtains that for every subset  $J\subset\mathbb{J}$ and $f\in\Hil$,

\begin{equation}\label{pop21}
\sum_{j\in{J}}|\langle f, e_j\rangle|^2-\|\sum_{j\in{J}}\langle{f}, e_j\rangle{e_j}\|^2=\sum_{j\in{J^c}}|\langle f, e_j\rangle|^2-\|\sum_{j\in{J^c}}\langle{f}, e_j\rangle{e_j}\|^2, \qquad J^c=\mathbb{J}\setminus{J}.
\end{equation}
 Taking \eqref{uman63b} into account, choosing $\Delta=J\in\mathcal{B}(\mathbb{R})$ and setting $E_j=j$ (i.e., ${\sf E}={J}$), we can rewrite \eqref{pop21} as follows
$$
\langle{F_{{\sf E}, e}(\Delta)}f, f\rangle-\langle{F_{{\sf E}, e}}(\Delta)f, {F_{\sf E, e}(\Delta)}f\rangle=\langle{F_{{\sf E}, e}}(\Delta^c)f, f\rangle-\langle{F_{{\sf E}, e}}(\Delta^c)f, {F_{{\sf E}, e}(\Delta^c)}f\rangle, 
$$
or, that is equivalent, $
\langle{F_{{\sf E}, e}}(\Delta^c){F_{{\sf N}, e}}(\Delta)f, f\rangle=\langle{F_{{\sf E}, e}(\Delta)}{F_{{\sf E}, e}}(\Delta^c)f, f\rangle.$ The last relation is obviously true since ${F_{{\sf E}, e}}(\Delta^c)=I-{F_{{\sf E}, e}}(\Delta)$.
This means that the fundamental identity \eqref{pop21} is, in fact, a straightforward property of POVMs. Similarly, the inequality 
$$
\sum_{j\in{J}}|\langle f, e_j\rangle|^2+\|\sum_{j\in{J^c}}\langle{f}, e_j\rangle{e_j}\|^2\geq\frac{3}{4}\|f\|^2
$$
established in \cite{BCE}, turns out to be obvious in POVMs language because it is equivalent to the obvious inequality $\|f-2{F_{{\sf E}, e}(\Delta)}f\|\geq{0}$.

\subsection{Commutative POVM}\label{sIII.2}

As is well known, the orthogonality of a POVM $F$
implies that
$F$ is commutative and, moreover, that
$F$ is a PVM. However, the commutativity of
$F$ does not necessarily imply that $F$ is a PVM. In this sense, commutative POVMs are a straightforward generalization of PVMs. Our objective is to characterize the commutative POVMs $F_{{\sf E},e}$ defined by \eqref{uman63b}.

In view of \eqref{auau1} with $A=1$ and \eqref{uman63b}, each POVM $F_ {{\sf E}, e}$ can be associated with the following decomposition of the identity as a sum of positive  bounded operators $L_k$: 
\begin{equation}\label{uman63c}
I=\sum_{k}L_k, \qquad L_k=\sum_{j\in\mathbb{J}_{\Delta_k}}\langle \cdot, e_j\rangle{e_j}, \qquad {\Delta_k}=\{E_k\},
\end{equation}
where different elements
$E_k\in{\sf E}$ determine different operators $L_k$. 

It follows from \eqref{uman63c} that the commutativity of  
$F_{{\sf E},e}$
 is equivalent to the mutual commutativity of the  operators 
$L_k$. For this reason, it is easier to find commutative POVMs when the number of distinct elements in ${\sf E}$ is small.  
In the simplest case, where ${\sf E}$ contains only one element $E$, the collection $\{L_k\}$
 reduces to the identity operator $I$. The  associated POVM 
$$
F_{{\sf E}, e}(\Delta)=\left\{\begin{array}{c}
0 \quad \mbox{if} \quad E\not\in\Delta; \\
I \quad \mbox{if} \quad E\in\Delta
\end{array}\right.
$$
is necessarily commutative for every choice of $\F_e$. 

Let us now consider the opposite case of POVM $F_{{\sf E}, e}$, where all elements of 
 ${\sf E}$ are pairwise distinct, i.e., $E_i \neq E_j$ whenever $i \neq j \in \mathbb{J}$. Then
 the set  $\{L_k\}_{k\in\mathbb{J}}$ consists of rank-one operators   $L_k=\langle \cdot, e_k\rangle{e_k}$.
 The commutativity of $F_{{\sf E}, e}$  is equivalent to the condition $[L_i, L_j]=0$.
This, in turn, is equivalent to the commutativity of the associated rank-one orthogonal projection operators
 $[P_i, P_j]=0$, where 
$$
P_k=\frac{1}{\|e_k\|^2}\langle \cdot, e_k\rangle{e_k}, \qquad k\in\{i, j\}.
$$
In view of the commutation relation, the operator
$P_iP_j$ is the orthogonal projection in $\Hil$ into the one-dimensional subspace spanned by both
$e_i$ and $e_j$ or $P_iP_j=0$. The first alternative corresponds to the case where the vectors 
$e_i$ and $e_j$ are collinear, while the second corresponds to  the case where the vectors 
$e_i$ and $e_j$ are orthogonal. Based on this observation, one can characterize commutative POVM's 
$F_{{\sf E},e}$.
\begin{thm}\label{uuu4}
Let $F_{{\sf E}, e}$ be the POVM defined in \eqref{uman63b}, where ${\sf E}$ is a set of distinct numbers.  The following are equivalent:
\begin{itemize}
\item[(i)] $F_{{\sf E}, e}$ is a commutative POVM;
\item[(ii)] there exists an orthonormal basis $\{\gamma_i\}$ of $\Hil$ such that  
\begin{equation}\label{kuku1}
\F_{e}=\{c_1^k\gamma_1\}_{k=1}^{r_1}\cup\{c_2^k\gamma_2\}_{k=1}^{r_2}\cup\{c_3^k\gamma_3\}_{k=1}^{r_3}\cup\ldots,
\end{equation}
where
\begin{equation}\label{kuku2}
\sum_{k=1}^{r_i}|c_i^k|^2=1, \qquad  c_i^k\not=0, \qquad 1\leq{i}\leq\dim\Hil, \qquad  r_i\in\mathbb{N}\cup\{\infty\}.
\end{equation}
\end{itemize}
\end{thm}
\begin{proof}
Set $\gamma_1=\frac{e_1}{\|e_1\|}$, $e_1\in\F_{e}$. As follows from the  above, the set $\F_{e}$ can be partitioned into two mutually disjoint subsets, denoted $\F^1$   and 
$\F^2$, where 
$\F^1$
  consists of all vectors 
$e_{j_1}^1,e_{j_2}^1, \ldots, e_{j_{r_1}}^1$ 
 that are collinear with 
$\gamma_1$, and 
$\F^2$ consists of all vectors $e_j\in\F_{e}$
 that are orthogonal to
$\gamma_1$. By construction $\F^1=\{e_{j_k}^1=c_1^k\gamma_1\}_{k=1}^{r_1}$, where $|c_1^k|=\|e_{j_{k}}^1\|$, $1\leq{k}\leq{r_1}$. The first relation in \eqref{kuku2} for the coefficients $\{c_1^k\}_{k=1}^{r_1}$ follows from the fact that $\{e_{j_k}^1\}_{k=1}^{r_1}$ is a Parseval frame in the one dimensional Hilbert space generated by $\gamma_1$. Choosing a vector
$\gamma_2=\frac{e_j}{\|e_j\|}$ where $e_j\in\F^2$ and repeating the same reasoning, we select the subset of vectors $\{e_{j_k}^2=c_2^k\gamma_2\}_{k=1}^{r_2}$ from $\F_e$ that are collinear with $\gamma_2$. Repeating this argument $\dim\Hil$-times completes the proof.
\end{proof}

Every commutative POVM $F$ admits a representation as a smearing of a PVM via Feller–Markov kernels \cite{RB1, RB2}.
This means
that the following formula holds:
\begin{equation}\label{uuu8}
F(\Delta)=\int_{\Gamma}\mu_{\Delta}(\lambda)dE(d\lambda), \qquad \Delta\in\mathcal{B}(\mathbb{R}),
\end{equation}
where $E$ is a PVM of a bounded self-adjoint operator $H$ with $\sigma(H)\subset[0, 1]$, $\Gamma\subset\sigma(H)$, $E(\Gamma)=I$,  and $\mu_\Delta(\lambda): \Gamma\times\mathcal{B}(\mathbb{R})\to[0,1]$ is a Feller-Markov kernel (see \cite{RB} --\cite{RB2} for details). The operator $H$ is called the \emph{sharp version} of $F$.

The orthonormal basis
$\{\gamma_i\}_{i=1}^{\dim\Hil}$ constructed in Theorem \ref{uuu4} makes it possible to explicitly construct sharp versions of commutative POVMs $F_{{\sf E}, e}$.

\begin{cor}\label{uuu7}
Under the assumptions of Theorem \ref{uuu4}, the sharp version  of a commutative POVM $F_{{\sf E}, e}$ has the form
\begin{equation}\label{uuu85}
H=\sum_{i=1}^{\dim\Hil}\lambda_i\langle\cdot, \gamma_i\rangle\gamma_i,
\end{equation}
where $\{\gamma_i\}_{i=1}^{\dim\Hil}$ is the orthonormal basis of $\Hil$ from \eqref{kuku1} 
and $\{\lambda_i\}_{i=1}^{\dim\Hil}$
 is an arbitrary sequence of distinct numbers satisfying 
$0\leq\lambda_i\leq{1}$.
\end{cor}
\begin{proof} Since $F_{{\sf E}, e}(\Delta)=F_{{\sf E}, e\not=0}(\Delta)$, we may assume, without loss of generality, that  $\F_e$ consists of nonzero vectors. 

Let 
$\Delta=\{E_j\}$, $j\in\mathbb{J}$. In this case, taking into account the definition \eqref{uman63b} of 
$F_{{\sf E}, e}$,
$$
F_{{\sf E}, e}(\{E_j\})f=\langle{f}, e_j\rangle{e_j}=\|e_j\|^2\left\langle{f}, \frac{e_j}{\|e_j\|}\right\rangle\frac{e_j}{\|e_j\|}, \qquad f\in\Hil.
$$
 From the proof of Theorem \ref{uuu4} it follows that $\frac{e_j}{\|e_j\|}$ coincides with a certain vector $\gamma_i$ up to multiplication by a unimodular constant. Here, the index $i$ is uniquely defined by $j$, that is, $i=p(j)$. The function $p(\cdot)$ defines a surjection of $\mathbb{J}$ onto the set  $\{1, 2, \ldots, \dim\Hil\}$. Using \eqref{kuku1}, we get
\begin{equation}\label{noch}
e_{j}=c_{p(j)}^{k(j)}\gamma_{p(j)}=c_{i}^{k(j)}\gamma_{i}, 
\end{equation}
where $|c_{p(j)}^{k(j)}|=\|e_{j}\|$ and the coefficient $1\leq{k(j)}\leq{r_{p(j)}}$ 
 specifies which vector $e_{j_1}^{p(j)}, e_{j_2}^{p(j)}, \ldots e_{j_{r_i}}^{p(j)},$ from the proof of Theorem \ref{uuu4}
 coincides with
$e_j$.
Finally, 
$$
F_{{\sf E}, e}(\{E_j\})f=|c_{p(j)}^{k(j)}|^2\langle{f}, \gamma_{p(j)}\rangle\gamma_{p(j)}.
$$

Denote by 
$\Gamma=\{\lambda_i\}_{i=1}^{\dim{\Hil}}$ an arbitrary sequence of distinct numbers $\lambda_i$ satisfying  $0\leq\lambda_i\leq{1}$
 and set 
 $$
\mu_{\{E_j\}}(\lambda)=\left\{\begin{array}{cc}
|c_{p(j)}^{k(j)}|^2  & \lambda=\lambda_{p(j)} \\
0 & \lambda\not=\lambda_{p(j)}
\end{array}\right.,  \qquad \lambda\in\Gamma.
$$
Denoting by
$E$  a
PVM of the self-adjoint operator $H$ defined by \eqref{uuu85}, we get
$$
\int_{\Gamma}\mu_{\{E_j\}}(\lambda)dE(d\lambda){f}=|c_{p(j)}^{k_j}|^2\langle{f}, \gamma_{p(j)}\rangle\gamma_{p(j)}=F_{{\sf E}, e}(\{E_j\})f.
$$
 The extension of $\mu_{\{E_j\}}(\lambda)$
 to the function 
$\mu_\Delta(\lambda) : \Gamma\times\mathcal{B}(\mathbb{R}) \to [0, 1]$:
\begin{equation}\label{uuu9}
\mu_\Delta(\lambda)=\sum_{E_j\in\Delta}\mu_{\{E_j\}}(\lambda), \qquad \lambda\in\Gamma
\end{equation}
 results in a Feller-Markov kernel (see, e.g., \cite[Definitions 7, 8]{RB}); we simply note that 
$\mu_\Delta(\lambda)$ is a probability measure for each $\lambda\in\Gamma$ in light of \eqref{kuku2}. Finally, recalling that
${\sf E}$ is a set of distinct numbers, we get the equality \eqref{uuu8} 
$$
F_{{\sf E}, e}(\Delta)f=\sum_{E_j\in\Delta}F_{{\sf E}, e}(\{E_j\})f=\sum_{E_j\in\Delta}\int_{\Gamma}\mu_{\{E_j\}}(\lambda)dE(d\lambda){f}=\int_{\Gamma}\mu_{\Delta}(\lambda)dE(d\lambda){f}.
$$
To complete the proof, it suffices to note that
the operator $H$ defined by \eqref{uuu85} is bounded, and $\sigma(H)\subset[0,1]$. 
\end{proof}

By means of Corollary \ref{uuu7}, the sharp version of a commutative POVM is not uniquely defined. We refer the reader to \cite{RN} for a detailed investigation of this issue.

\begin{example} Let $\Hil=\mathbb{C}^2$. Consider a set of distinct real numbers ${\sf E}=\{E_1, E_2, E_3\}$ and a Parseval frame $\F_e=\{e_1, e_2, e_3\}$ in $\mathbb{C}^2$, where
$$
e_1=\sqrt{\frac{2}{3}}\left[\begin{array}{c}
1 \\
0 \end{array}\right], \qquad e_2=\sqrt{\frac{1}{3}}\left[\begin{array}{c}   
1 \\
0 \end{array}\right], \qquad e_3=\left[\begin{array}{c}
0 \\
1 \end{array}\right].
$$
By virtue of Theorem \ref{uuu4},  the POVM $F_{{\sf E}, e}$ defined by \eqref{uman63b} is commutative. The corresponding orthonormal basis $\{\gamma_i\}_{i=1}^2$ coincides with the canonical basis of $\mathbb{C}^2$.
The sharp version of $F_{{\sf E}, e}$ is given by a self-adjoint operator $$
H=\lambda_1\langle\,\cdot,\gamma_1\rangle \gamma_1+\lambda_2\langle\cdot, \gamma_2\rangle\gamma_2=\left[\begin{array}{cc}
\lambda_1 & 0 \\
0 & \lambda_2 \end{array}\right],
$$ 
where $0\leq\lambda_1\neq\lambda_2\leq{1}$. The Feller-Markov kernel is defined by \eqref{uuu9}, where $\mu_{\{E_1\}}(\lambda_1)=\frac{2}{3}$, $\mu_{\{E_1\}}(\lambda_2)=0$, $\mu_{\{E_2\}}(\lambda_1)=\frac{1}{3}$, $\mu_{\{E_2\}}(\lambda_2)=0$, $\mu_{\{E_3\}}(\lambda_1)=0$, $\mu_{\{E_3\}}(\lambda_2)=1.$
\end{example}

\subsection{Jointly measurable  POVMs}\label{sIII.3}
POVMs $F_j : \mathcal{B}(\mathbb{R}) \to \mathcal{L}^+(\Hil), \ j=1,2$ are called \emph{jointly measurable} if they are the marginals of a joint POVM $F : \mathcal{B}(\mathbb{R}\times\mathbb{R}) \to \mathcal{L}^+(\Hil)$. Here, the symbol $\mathcal{B}(\mathbb{R}\times\mathbb{R})$ denotes the product $\sigma$-algebra generated by the family of sets $\{\Delta_1\times\Delta_2 : \Delta_j\in\mathcal{B}(\mathbb{R})\}$.  

Two POVMs $F_j : \mathcal{B}(\mathbb{R}) \to \mathcal{L}^+(\Hil)$  \emph{commute} if 
$[F_1(\Delta_1), F_2(\Delta_2)]=0$ for each $\Delta_j\in\mathcal{B}(\mathbb{R}), \ j=1,2$. 
In the following, the commutativity of two POVMs is indicated by the symbol $[F_1, F_2]=0$.  

If $F_1$ and $F_2$ are  PVMs, they are jointly measurable if and only if $[F_1, F_2]=0$. In the case of POVMs, commutativity implies joint measurability, but the converse is not true \cite{Lachti}. Our objective is to determine under what conditions POVMs defined by \eqref{uman63b} are jointly measurable.

 Let $F_{{\sf E}, e}$ and $F_{\widehat{\sf E}, \widehat{e}}$ be two POVMs where $\F_{e}$ and $\F_{\widehat{e}}$ are Parseval frames in $\Hil$ with the same excess number ${\bf e}[\F_e]={\bf e}[\F_{\widehat{e}}]$ and with the same index set $\mathbb{J}$. According to the Naimark theorem \ref{Naimark}, these Parseval frames can be extended to orthonormal bases $\F_{e^+}$ and $\F_{\widehat{e}^+}$ of the same Hilbert space $\Hil^+$ (see \eqref{ups1}). This allows one to construct the PVM
 \begin{equation}\label{uman63f}
F_ {{\sf E}, e^+}(\Delta)=\sum_{j\in\mathbb{J}_\Delta}\langle \cdot, e_j^+\rangle e_j^+, \qquad \mathbb{J}_\Delta=\{j\in\mathbb{J} : E_j\in\Delta\}, \qquad \Delta\in\mathcal{B}(\mathbb{R})
\end{equation}
and, similarly, the PVM $F_{\widehat{\sf E}, \widehat{e}^+}$ with the use of  $\F_{\widehat{e}^+}$ and  $\widehat{E}=\{E_j\}_{j\in\mathbb{J}}$.

By virtue of 
\cite[Theorem 6]{RB}, 
POVMs $F_{{\sf E}, e}$ and $F_{\widehat{\sf E}, \widehat{e}}$ are 
jointly measurable if and only if $[F_{{\sf E}, e^+},  F_{\widehat{\sf E}, \widehat{e}^+}]=0$. 
If $\F_{e}=\F_{\widehat{e}}$, then the orthonormal bases $\F_{e^+}$ and $\F_{\widehat{e}^+}$ coincide and the last commutation relation is evidently valid.
The following result is now proved 
\begin{prop}
 POVMs $F_{{\sf E}, e}$ and $F_{\widehat{\sf E}, e}$ constructed using the same Parseval frame $\F_e$ and arbitrary sets ${\sf E}$ and $\widehat{\sf E}$ are jointly measurable.
\end{prop}

Consider another special case, assuming $\F_{e}\not=\F_{\widehat{e}}$, and assuming that both ${\sf E}$ and
$\widehat{\sf E}$ contain distinct numbers. Then the relation $[F_{{\sf E}, e^+},  F_{\widehat{\sf E}, \widehat{e}^+}]=0$, which ensures the joint measurability of $F_{{\sf E}, e}$ and $F_{\widehat{\sf E}, \widehat{e}}$ (see \cite[Theorem 6]{RB}), is equivalent to the relation $[L_j, \widehat{L}_i]=0$ for all possible pairs of indices $j,i\in\mathbb{J}$, where
$$
L_j=\langle\cdot, e^+_j\rangle{e_j^+}, \qquad \widehat{L}_i=\langle\cdot, \widehat{e}^+_i\rangle{\widehat{e}_i^+}.
$$
Arguing as in the proof of Theorem \ref{uuu4} and taking into account that the sets $\{e_j^+\}$ and $\{\widehat{e}^+_i\}$ are orthonormal bases of $\Hil^+$, we conclude that $[L_j, \widehat{L}_i]=0$ holds for all indexes $j,i\in\mathbb{J}$ if and only if
there exists bijection $j \to i$ on $\mathbb{J}$ such that each vector $e_j^+$ coincides with $\widehat{e}_i^+$ up to multiplication by a unimodular constant. Therefore (see Sect. \ref{sIII.1}), we can assume 
$e_j^+=\widehat{e}_i^+$  without loss of generality. In this case, as follows from \eqref{ups1},  $e_j=\widehat{e}_i$. Therefore, in the case under consideration, \emph{joint measurability is possible only if both POVM's
$F_{{\sf E}, e}$ and $F_{\widehat{\sf E}, \widehat{e}}$
 are constructed based on the same Parseval frame $\F_e=\F_{\widehat{e}}$.}

\section{Operators generated by Parseval frames}\label{sIV}
\subsection{POVM and associated operators}\label{sIV.1}
Given a POVM $F : \mathcal{B}(\mathbb{R}) \to \mathcal{L}^+(\Hil)$. 
\emph{Is it possible to construct a self-adjoint operator 
$H$ on $\Hil$ that arises naturally from $F$?} If $F$ is a PVM, then a self-adjoint operator $H$ is defined in a simple and natural form \begin{equation}\label{uman63gg}
H=\int_{\mathbb{R}}\lambda{d} F(d\lambda), \qquad \mathcal{D}(H)=\{ f\in\Hil \ : \ \int_{\mathbb{R}}\lambda^2d\langle F(d\lambda)f, f \rangle<\infty\}.
\end{equation}
Here, the spectrum of $H$
coincides with the spectrum of $F$.

If $F$ is a POVM, one can use the Naimark theory approach \cite{Nai}, where the primary attention is paid to the preservation of formulas \eqref{uman63gg}. In this context (see \cite[\$ 111]{AG} for details), each POVM $F$ corresponds to a symmetric operator $H$, for which the following identities are held for $f\in\mathcal{D}(H)$ and  $g\in\Hil$
$$
\langle Hf, g \rangle=\int_{\mathbb{R}}\lambda{d(\langle F(d\lambda)f, g \rangle}, \qquad  
\| Hf\|^2=\int_{\mathbb{R}}\lambda^2{d\langle F(d\lambda)f, f \rangle}.
$$
Let us examine the case where $F=F_{{\sf E}, e}$. By virtue of Theorem \ref{Naimark} and \eqref{uman63b},
\begin{equation}\label{uman63d}
F_ {{\sf E}, e}(\Delta)f=P^+F_ {{\sf E}, e^+}(\Delta)f, \qquad f\in\Hil, 
\end{equation}
where $P^+$ is an orthogonal projection operator on $\Hil^+$ in $\Hil$ and  $F_ {{\sf E}, e^+}$ is the PVM  constructed by the orthonormal basis $\F_{e^+=}\{e_j^+\}_{j\in\mathbb{J}}$ of $\Hil^+$ described in Theorem \ref{Naimark} and the set ${\sf E}$, see \eqref{uman63f}.
In view of \eqref{uman63d}
\begin{equation}\label{busa1}
\langle Hf, g\rangle=\int_{\mathbb{R}}{\lambda}d\langle F_{{\sf E}, e}(d\lambda)f, g \rangle=\int_{\mathbb{R}}{\lambda}d\langle F_{{\sf E}, e^+}(d\lambda)f, g \rangle=\langle{H_{{\sf E},e^+}f, g\rangle}  
\end{equation}
and
\begin{equation}\label{busa2}
\|Hf\|^2=\int_{\mathbb{R}}{\lambda}^2d\langle F_{{\sf E}, e}(d\lambda)f, f\rangle=\int_{\mathbb{R}}{\lambda}^2d\langle F_{{\sf E}, e^+}(d\lambda)f, f \rangle=\|H_{{\sf E}, e^+}f\|^2,
\end{equation}
where 
$$    H_{{\sf E},e^+}=\sum_{j\in\mathbb{J}}E_j \langle \cdot,  e_j^+ \rangle{e}_j^+, \qquad  \mathcal{D}(H_{{\sf E}, e^+})=\{f\in\Hil^+ \ : \ \sum_{j\in\mathbb{J}}E_j^2 |\langle f, e_j^+\rangle|^2<\infty\}
$$
is a self-adjoint operator in $\Hil^+$ associated with the PVM $F_{{\sf E}, e^+}$ via \eqref{uman63gg}.

It follows from \eqref{busa1} and \eqref{busa2}  that  the symmetric operator $H$ has the form
$$
H=H_{{\sf E}, e^+}, \qquad \mathcal{D}(H)=\{f\in\mathcal{D}(H_{{\sf E}, e^+})\cap\Hil \ : \ H_{{\sf E}, e^+}f\in\Hil \},
$$
or, equivalently, using \eqref{ups1},
\begin{equation}\label{busa3}
H=\sum_{j\in\mathbb{J}}E_j \langle \cdot,  e_j \rangle{e}_j, \quad \mathcal{D}(H)=\{f\in\Hil \ : \ \sum_{j\in\mathbb{J}}E_j^2 |\langle f, e_j\rangle|^2<\infty, \ \sum_{j\in\mathbb{J}}E_j \langle f, e_j\rangle{m_j}=0\},
\end{equation}
where
$\F_m=\{m_j\}_{j\in\mathbb{J}}$ is a complementary Parseval frame in Theorem \ref{Naimark}.
 However, such a definition is not suitable for our purposes, since the resulting operator
$H$ is typically not densely defined. 

\begin{example}\label{e8}
Consider the so-called Mercedes frame $\F_e=\{e_1, e_2, e_3\}$ in $\Hil=\mathbb{C}^2$, 
$$
e_1=\sqrt{\frac{2}{3}}\left[
\begin{array}{c}
1 \\
0
\end{array}
\right], \qquad e_2=\sqrt{\frac{2}{3}}\left[
\begin{array}{c}
-\frac{1}{2} \\
\frac{\sqrt{3}}{2} 
\end{array}
\right], \qquad e_3=\sqrt{\frac{2}{3}}\left[
\begin{array}{c}
-\frac{1}{2} \\
 -\frac{\sqrt{3}}{2}
\end{array}
\right].
$$
The vectors of the Parseval frame $\F_m=\{m_1, m_2, m_3\}$ in Theorem \ref{Naimark} have the form (see \cite{BK1})
$$
m_1=m_2=m_3=\frac{1}{\sqrt{3}}\left[\begin{array}{c}
0 \\
0 \\
1 \end{array}\right].
$$
In view of \eqref{busa3}, the operator   
$$
H=\sum_{j=1}^3E_j \langle \cdot,  e_j \rangle{e}_j=\frac{1}{6}\left[\begin{array}{cc}
4E_1+E_2+E_3 & \sqrt{3}(E_3-E_2) \\
\sqrt{3}(E_3-E_2) & 3(E_2+E_3)\end{array}
\right]
$$
acting in $\mathbb{C}^2$ has the domain of definition $\mathcal{D} (H)$ that consists of vectors $f=\left[\begin{array}{c} x\\
y \end{array}\right]$ for which $\frac{1}{\sqrt{3}}\sum_{j=1}^3E_j \langle f, e_j\rangle=0$ or, equivalently, solutions of the equation
$$
(2E_1-E_2-E_3)x+\sqrt{3}(E_2-E_3)y=0.
$$
Consequently, the operator 
$H$
 is defined on $\mathbb{C}^2$ when $E_1=E_2=E_3$. Otherwise, the domain of 
$H$ is not dense.
\end{example}

Working toward resolving issues related to non-densely defined domains, we may associate with 
POVM $F_{{\sf E}, e}$
 an extended version of \eqref{busa3}, omitting the condition $\sum_{j\in\mathbb{J}}E_j \langle f, e_j\rangle{m_j}=0$:
 \begin{equation}\label{K5}
H_{}=P^+H_{{\sf E}, e^+}=\sum_{j\in\mathbb{J}}E_j \langle  \cdot, e_j \rangle\,e_j, 
\end{equation}
\begin{equation}\label{busa4}
 \mathcal{D}(H)=\mathcal{D}(H_{\sf E, e^+})\cap{\mathcal{H}}=\{f\in\mathcal{H} : \sum_{j\in\mathbb{J}}E_j^2|\langle f, e_j\rangle|^2<\infty \}.  
\end{equation}
At this point, a natural question arises: is the definition of 
$H$, given by the formulas 
\eqref{K5} and 
\eqref{busa4}, optimal? 
 The following example provides a hint to the proper answer.

\begin{example}\label{rrr1}
Let $\{\gamma_n\}_{n=1}^\infty$ be an orthonormal basis of $\Hil$. Set $\mathbb{J}=\mathbb{Z}\setminus\{0\}$ and consider the quantities $E_n=(n+1)^2, \ E_{-n}=0$ ($n\in\mathbb{N}$). The set of vectors $\F_e=\{e_j\}_{j\in\mathbb{J}}$, where
$$
e_n=\frac{1}{n+1}\gamma_n, \qquad e_{-n}=\sqrt{1-\frac{1}{(n+1)^2}}\gamma_n, \qquad n\in\mathbb{N}
$$
is a Parseval frame in $\Hil$. The operator $H$ defined by \eqref{K5} coincides with the 
identity operator
$$
H=\sum_{j\in\mathbb{J}}E_j\langle\cdot, e_j\rangle{e_j}=\sum_{n\in\mathbb{N}}\langle\cdot, \gamma_n\rangle{\gamma_n}=I
$$
but, in view of \eqref{busa4}, it is defined on the dense set
$\{f\in\Hil  :  \sum_{n\in\mathbb{N}}{(n+1)^2}|\langle{f, \gamma_n}\rangle|^2<\infty\}$ in $\Hil$.
\end{example}

The difficulty in constructing a self-adjoint operator, as presented in the above example, can be resolved by extending the domain of definition of  $H$. It suffices to consider the domain
\begin{equation}\label{busa4b}
\mathcal{D}(H)=\{f\in\Hil \ : \ \sum_{j\in\mathbb{J}}E_j\langle{f}, e_j\rangle{e_j} \ \mbox{converges unconditionally in }\ \Hil\}
\end{equation} in place of \eqref{busa4}.
In general, the set defined by \eqref{busa4} is smaller than the set determined by \eqref{busa4b}. However, for many frames, these two expressions define the same set of vectors. For example, this  holds when the excess ${\bf e}[\F_e]$ of 
$\F_e$ is finite or when the set 
${\sf E}$ is bounded, that is, $\sup_{j\in\mathbb{J}}\{|E_j|\}<\infty$ \cite{BK2}.
Nevertheless, in the case where 
${\bf e}[\F_e]=\infty$ and
$\sup_{j\in\mathbb{J}}\{|E_j|\}=\infty$, the situation becomes considerably more delicate. Specifically, for any unbounded set 
${\sf E}$, there exists a Parseval frame $\F_e$  such that the domain \eqref{busa4}
 (or the domain \eqref{busa4b}) 
produces a trivial set\footnote{The corresponding result for 
\eqref{busa4}
 was established in \cite{BK2}; the argument for \eqref{busa4b} 
is analogous.} 
$\mathcal{D}(H)=\{0\}$.
However, when the Parseval frame (or, equivalently, the corresponding POVM) satisfies certain structural conditions, it is possible to preserve the self-adjointness of the corresponding operator $H$.

In what follows, taking formulas \eqref{uman63r} and \eqref{K5} into account, we denote by 
$H_{{\sf E},e}$ the operator defined in 
\eqref{K5}. In summary, we will consider the operator 
$H_{{\sf E},e}$ defined by 
\eqref{K5}
(or equivalently \eqref{uman63r}), with two possible domain choices, namely 
\eqref{busa4}
 or \eqref{busa4b}.

\begin{thm}\label{duda1}
Assume that a POVM $F_{{\sf E},e}$ is commutative and ${\sf E}$ contains distinct numbers. Then the 
operator $H_{{\sf E},e}$ with the domain  \eqref{busa4b} is self-adjoint in $\Hil$ and its spectrum coincides with the closure of eigenvalues 
\begin{equation}\label{eigen12}
    \left\{\lambda_i=\sum_{k=1}^{r_i}E^{k}_i|c_i^k|^2\right\}_{i=1}^{\dim\Hil},
\end{equation}
where the coefficients $c_i^k$ are taken from \eqref{kuku1} and $E^{k}_i\in{\sf E}$ corresponds to the vector $e_j=c_i^{k(j)}\gamma_i\in\F_e$ in \eqref{noch}.
\end{thm}
\begin{proof}  Assume that $f$ belongs to the domain defined by \eqref{busa4b}.
Then, employing  \eqref{kuku1}, one can rewrite \eqref{K5} as follows:
\begin{equation}\label{fufa1}
H_{{\sf E},e}f=\sum_{j\in\mathbb{J}}E_j \langle  f, e_j \rangle\,e_j=\sum_{i=1}^{\dim\Hil}\sum_{k=1}^{r_i}E^{k}_i|c_i^k|^2\langle  f, \gamma_i \rangle\,\gamma_i,  \end{equation}
where $\{\gamma_i\}$ is an orthonormal basis in $\Hil$.
Projecting onto the one-dimensional subspace generated by 
$\gamma_i$ yields a series 
$$
\left(\sum_{k=1}^{r_i}E^{k}_i|c_i^k|^2\right)\langle  f, \gamma_i \rangle\,\gamma_i
$$
 that converges unconditionally in 
$\Hil$, and hence 
$\sum_{k=1}^{r_i}E^{k}_i|c_i^k|^2$ also converges unconditionally in $\mathbb{C}$.
Denote
$\lambda_i=\sum_{k=1}^{r_i}E^{k}_i|c_i^k|^2$ and consider a self-adjoint operator in $\Hil$
$$
\widehat{H}=\sum_{i=1}^{\dim\Hil}\lambda_i\langle{\cdot}, \gamma_i\rangle\gamma_i, \qquad \mathcal{D}(\widehat{H})=\{f\in\Hil : \sum_{i=1}^{\dim\Hil}\lambda_i^2|\langle{f}, \gamma_i\rangle|^2<\infty\}.
$$
Since $\{\gamma_i\}$ is an orthonormal basis, the set $\mathcal{D}(\widehat{H})$ admits an equivalent description \cite[p. 200]{Heil}
$$
\mathcal{D}(\widehat{H})=\{f\in\Hil \ : \ \sum_{i=1}^{\dim\Hil}\lambda_i\langle{\cdot}, \gamma_i\rangle\gamma_i  \quad \mbox{converges unconditionally} \ \}.  
$$
By comparing the preceding relation with \eqref{fufa1} and using \eqref{kuku1}, we get that 
$\widehat{H}$
 coincides with the operator $H_{{\sf E},e}$ defined on the domain \eqref{busa4b}. This completes the proof.
\end{proof}

\begin{rem} If all $\lambda_i$ belong to $[0, 1]$, then the operator $H_{{\sf E},e}$ considered in Theorem \ref{duda1} coincides with the sharp version of POVM $F_{{\sf E}, e}$, see Corollary \ref{uuu7}.
\end{rem}

\subsection{$\sigma(F_ {{\sf E}, e})$ vs. $\sigma(H_{{\sf E},e})$: what they share and how they differ}\label{sIV.2}
In view of \eqref{uman63kk}, the spectrum 
$\sigma(F_{{\sf E}, e})$ of the POVM $F_{{\sf E}, e}$
coincides with the closure of 
${\sf E}$. Our aim is now to analyze the relationship between 
$\sigma(F_{{\sf E}, e})$ and $\sigma(H_{{\sf E}, e})$.

If the excess ${\bf e}[\F_e]$ of $\F_e$ is zero, then $\F_e$ turns out to be an orthonormal basis in $\Hil$ and $\sigma(F_{{\sf E}, e})=\sigma(H_{{\sf E}, e})$.
However, when $\mathbf{e}[\mathcal{F}_e] > 0$, the relationship between $\sigma(F_{{\sf E}, e})$ and $\sigma(H_{{\sf E}, e})$ becomes less obvious than that described above. 
The variety of possible relationships is easily illustrated in the case of commutative POVMs. Consider a commutative POVM as described in Theorem 8. Then,  the spectrum of the self-adjoint operator $H_{{\sf E}, e}$ with the domain \eqref{busa4b} coincides with the closure of eigenvalues \eqref{eigen12}. 

Assume that $\mathbf{e}[\mathcal{F}_e]$ is finite, let $\mathbf{e}[\mathcal{F}_e]=n$.
Then, among the vectors $\gamma_i$ in \eqref{kuku1}, at most 
$n$ can be repeated (that is, it has at least two nonzero coefficients 
$c_i^1, c_i^2$). Consequently, $\sigma(H_{{\sf E}, e})$
contains all elements of ${\sf E}$ except possibly 
$n$ of them.

Consider one of the simplest examples of Parseval frames with 
$\mathbf{e}[\mathcal{F}_e]=\infty$, assuming in \eqref{kuku1} that 
$\dim\Hil=\infty$ and that all basis vectors corresponding to $\gamma_i$ are repeated  $m$ times, i.e.,
$c_i^1=\ldots=c_i^m=\frac{1}{\sqrt{m}}$, \  $1\leq{i}\leq\infty$, \ $m\in\mathbb{N}$.
In this case, \eqref{eigen12} implies that the eigenvalues of $H_{{\sf E}, e}$ are given by weighted averages of the values $E_{i}^1, \ldots, E_{i}^m$ of ${\sf E}$:
$$
\left\{\lambda_i = \frac{1}{m}\sum_{k=1}^m E^{k}_i\right\}_{i=1}^\infty.
$$
Instead of assuming a constant repetition number $m$ and a constant multiplier $\frac{1}{\sqrt{m}}$, one may consider a variable choice of these parameters, chosen so that the coefficients of the basis vectors coincide with the square roots of the weights in a weighted average.
 
In the general case of non-commutative POVMs, as shown below, if a Parseval frame 
$\F_e$ possesses additional properties relevant to applications, one may expect  predictable relations between 
${\sf E}$
 and  $\sigma(H_{{\sf E},e})$.

 Before turning to concrete examples, we present a general scheme for the calculation of $\sigma(H_{{\sf E},e})$. To do this, we first develop a convenient description of $H_{{\sf E}, e}$.
Denote by
$$
\theta_e{f}=\{\langle f, e_j\rangle\}_{j\in\mathbb{J}}, \quad f\in\Hil \quad \mbox{and} \quad \theta_e^*\{c_j\}=\sum_{j\in\mathbb{J}}c_je_j, \quad \{c_j\}\in{l_2(\mathbb{J})}$$
\emph{the analysis operator} $\theta_e$ and \emph{the synthesis operator} $\theta_e^*$, associated with  $\F_e$.

For a Parseval frame, the analysis operator is an isometric mapping $\Hil \to l_2(\mathbb{J})$. Hence, $\theta_e^*\theta_e=I$,
where $I$ is the identity operator in $\Hil$. On the other hand, $\theta_e\theta_e^*$ is the orthogonal projection of 
$l_2(\mathbb{J})$ into the range 
$\mathcal{R}(\theta_e)$ of $\theta_e$. The orthogonal projection operator
$\theta_e\theta_e^*$ in $l_2(\mathbb{J})$ can be represented as multiplication by the Gram matrix $G_e=\{\langle e_k, e_j\rangle\}_{j,k\in\mathbb{J}}$ of $\F_e$ (see \cite[p. 191]{Heil} for details). In what follows, we identify the projection operator   $\theta_e\theta_e^*$ and the Gram matrix $G_e$ that represents it.

Consider an operator of multiplication by the set ${\sf E}=\{E_j\}_{j\in\mathbb{J}}$ in $l^2(\mathbb{J})$: 
$$
\mathcal{E}\{c_j\}=\{E_jc_j\}, \qquad \mathcal{D}(\mathcal{E})=\{\{c_j\}\in{l^2(\mathbb{J}) \ : \ \{E_jc_j\}\in{l^2(\mathbb{J})} }\}
$$
and rewrite the formula \eqref{K5}
with the use of
$\mathcal{E}$,
\begin{equation}\label{xoxo1}
H_{{\sf E}, e}=\theta^*_e\mathcal{E}\theta_e=\theta^*_e\theta_e\theta^*_e\mathcal{E}\theta_e\theta^*_e\theta_e=\theta^*_eG_e\mathcal{E}G_e\theta_e.
\end{equation}

\begin{thm}\label{xx36}  Assume that the operator
$H_{{\sf E}, e}$ with the domain of definition \eqref{busa4} is densely defined in $\Hil$. Then  
 $\sigma(H_{{\sf E}, e})\setminus\{0\}$ coincides with $\sigma(B)\setminus\{0\}$, where 
 \begin{equation}\label{xx2}
B=G_e\mathcal{E}G_e, \qquad \mathcal{D}(B)=\{\{c_j\}\in{l_2(\mathbb{J})} \ : \ G_e\{c_j\}\in\mathcal{D}(\mathcal{E})\}
 \end{equation}
 is an operator 
 acting in $l_2(\mathbb{J})$.
\end{thm}
\begin{proof} It follows from \eqref{xx2} that $\ker{B}\supseteq\ker\theta_e^*$ and the operator $B$ is decomposed into $B=B|_{\mathcal{R}(\theta_e)}\oplus{0}$ with respect to the decomposition\footnote{$\mathcal{R}(L)$ and
$\ker{L}$, denote the \emph{range} and the \emph{null space} of an operator $L$, respectively. 
} 
$$
l_2(\mathbb{J})=\mathcal{R}(\theta_e)\oplus\ker\theta_e^*.
$$
Here, $B|_{\mathcal{R}(\theta_e)}$ means the restriction of $B$ on $\mathcal{R}(\theta_e)$. It follows from \eqref{busa4} and \eqref{xx2} that
$$
\mathcal{D}(B|_{\mathcal{R}(\theta_e)})=\{\{c_j\}\in\mathcal{R}(\theta_e) \ : \ \{c_j\}\in\mathcal{D}(\mathcal{E})\}=\theta_e\mathcal{D}(H_{{\sf E}, e})
$$
and $\mathcal{D}(B)=\mathcal{D}(B|_{\mathcal{R}(\theta_e)})\oplus\ker\theta_e^*$.
Taking \eqref{xoxo1} into account, we find that
$B|_{\mathcal{R}(\theta_e)}$ is isometrically isomorphic to the operator $H_{{\sf E}, e}$ with domain \eqref{busa4}. This concludes the proof.  \end{proof}

The case of 
$0\in\sigma(H_{{\sf E}, e})$ can also be described, but it requires an additional condition, since zero is always an eigenvalue of 
$B$. The situation in finite-dimensional spaces is particularly simple.
The following result is a direct consequence of Theorem \ref{xx36}.

\begin{cor}\label{xoxo35}  If $\dim\Hil<\infty$, then  $\sigma(H_{{\sf E}, e})\setminus\{0\}$ coincides with the set of non-zero eigenvalues of the matrix $G_e\mathcal{E}G_e$. Furthermore, $0\in\sigma(H_{{\sf E}, e})$ if and only if the eigenvalue $0$ of $G_e\mathcal{E}G_e$ has multiplicity greater than ${\bf e}[\F_e]=|\mathbb{J}|-\dim\Hil$.
\end{cor}

\subsection{Examples}\label{sIV.3}

\subsubsection{Mercedes-type Parseval frames}
A frame $\F_\psi=\{\psi_j, j\in\mathbb{J}\}$ is called a \emph{dual frame} for a PF $\F_e=\{e_j, j\in\mathbb{J}\}$ if 
$$
f=\sum_{j\in\mathbb{J}}\langle f, e_j\rangle \psi_j=\sum_{j\in\mathbb{J}}\langle f, \psi_j\rangle e_j, \qquad f\in\Hil.
$$

There are multiple techniques for building dual frames, as described in  \cite{chri, onDF, Can}. 
One of the most frequently used expressions is the formula \cite[p.159]{chri}
$$\F_{\psi} =\left\{\psi_j= e_{j}+h_{j}-\sum_{n\in \mathbb{J}}\langle e_{j},e_{n}\rangle h_{n}\right\},
$$ 
where $\{h_{n}\}_{n\in\mathbb{J}}$ is a Bessel sequence in $\Hil$. Applying this formula to the Mercedes frame $\F_e=\{e_1, e_2, e_3\}$
 considered in Example \ref{e8} shows that all dual frames of 
$\F_e$  can be constructed in an extremely simple way: by adding an arbitrary vector $h\in\Hil=\mathbb{C}^2$  to each vector of $\F_e$. Specifically,
$\F_\phi=\{\phi_j=e_j+h\}_{j=1}^3.$ 
Taking this property as characteristic of the Mercedes frame, we can introduce a generalized version of the Mercedes frame. That is, we say that a Parseval frame $\F_e=\{e_j\}_{j\in\mathbb{J}}$ in $\Hil$  is of \emph{Mercedes type} if at least one of its dual frames has the form 
$\F_\psi=\{\psi_j=e_j+h\}_{j\in\mathbb{J}}$, where $h\not=0\in\Hil$.

A simple  analysis shows that 
$\F_e$ is a Mercedes-type frame if and only if $3\leq|\mathbb{J}|<\infty$, where $k=|\mathbb{J}|$  
and the Gram matrix of $\F_e$ is a $k\times{k}$-matrix of the form 
\begin{equation}
G_e=\frac{1}{k} 
   \left[ \begin{array}{cccc}
    k-1 & -1&\dots& -1\\
   -1& k-1& \dots&-1 \\
   \vdots & &  \ddots\\
    -1& \dots&-1&k-1
 \end{array}\right].
 \label{gramm}
\end{equation}

\begin{prop}\label{uuu1} 
Each Mercedes-type Parseval frame $\F_e=\{e_j\}_{j\in\mathbb{J}}$ has excess ${\bf e}[\F_e]=1$ and can be rescaled to the optimal  Grassmannian frame 
$$
\F_{e'}=\left\{{e_j}'=\sqrt{\frac{k}{k-1}}e_j\right\}_{j\in\mathbb{J}}, \qquad k=|\mathbb{J}|.
$$
\end{prop}
\begin{proof} It follows from \eqref{gramm} that $\F_{e'}$ is an optimal Grassmannian frame with frame bound $A=\frac{k}{k-1}$. By comparing this relation with formula $A=\frac{|\mathbb{J}|}{|\mathbb{J}|-{\bf e}[\F_{e'}]}$
 in Sect. \ref{sII.1} and using the fact that ${\bf e}[\F_{e'}]={\bf e}[\F_{e}]$, we obtain ${\bf e}[\F_{e}]=1$. 
\end{proof}

The proof of the statement below is given in the Appendix.
\begin{thm}\label{ttt63}
Let $\F_e=\{e_j\}_{j\in\mathbb{J}}$
 denote a Mercedes-type Parseval frame in the Hilbert space 
$\Hil$, and let $H_{{\sf E}, e}$ 
 be the operator on $\Hil$ defined by \eqref{K5}. Then 
$H_{{\sf E}, e}$ has eigenvalues $$
\lambda_1, \lambda_2, \ldots, \lambda_{k-1}, \qquad k=|\mathbb{J}|,
$$
given by the roots of the equation
$
\sum_{i=1}^k \prod_{j\neq i}(E_j - \lambda)=0. 
$
\end{thm}
\begin{example}
For the Mercedes frame $\F_e=\{e_1, e_2, e_3\}$ in $\Hil=\mathbb{C}^2$ considered in Example \ref{e8}, $k=3$, and the roots $\lambda_1, \lambda_2$ of the equation
$$
(E_1-\lambda)(E_2-\lambda)+(E_1-\lambda)(E_3-\lambda)+(E_2-\lambda)(E_3-\lambda)=0
$$
are eigenvalues of  $H_{{\sf E, e}}=\sum_{j=1}^3E_j \langle \cdot,  e_j \rangle{e}_j$. 
\end{example}

\subsubsection{Parseval frames  arising from conference matrices}
Another class of Parseval frames related to optimal Grassmannian frames can be constructed using the so-called conference matrices. A conference matrix 
$C_N$ is a 
$N\times N$ matrix with entries in $\{-1,0,1\}$ satisfying  $C_N^TC_N = (N-1)I_N$. 
Conference matrices can be viewed as Seidel adjacency matrices of graphs and were first introduced in connection with a problem in telephony \cite{BV}.  
They also belong to the broader class of weighing matrices which arise, for example, in statistics \cite{OD}.

It turns out that, given a conference matrix $C_N$ 
 with 
$N=2m$, one can explicitly construct an optimal Grassmannian frame 	
for $\mathbb{R}^m$. The proof of the result below is given in \cite{SH}. \begin{thm}\label{ttt1}
     If $N=2m = p^l+1$ for some prime $p$ and positive integer $l$, then an optimal Grassmannian frame $\mathcal{F}_{e'}=\{e'_j\}_{j\in\mathbb{J}}$ for $\mathbb{R}^m$ with excess $\boldsymbol{e}[\mathcal{F}_{e'}] = m$ and $|\mathbb{J}|=N$ can be explicitly constructed.
\end{thm}
In case of such frames we have $\alpha=|\langle e'_j,e'_i\rangle| = \frac{1}{\sqrt{N-1}}$ (see Sect. \ref{sII.1}) and the Gram matrix $G_{e'}$ for $\mathcal{F}_{e'}$ can be derived from the conference matrix $C_N$ by setting $G_{e'} = \alpha C_N + I$. 

Consider the case $p=5$ and $l=1$. 
By Theorem~\ref{ttt1}, in 
$\Hil=\mathbb{R}^3$
 there exists an optimal Grassmannian frame $\F_{e'}=\{e'\}_{j=1}^6$
 with coherence constant $\alpha = \frac{1}{\sqrt{5}}$. 
This frame
 can be constructed by normalizing the vectors $f_j$, that is, setting 
$e'_j = \frac{1}{\|f_j\|} f_j$,
where
$$
\begin{array}{ccc}
f_1  = (\beta, 1, 0), & f_2 = (\beta, -1, 0), & f_3 = (1, 0, \beta) \\
f_4 = (-1, 0, \beta), & f_5 = (0, \beta, 1), & f_6 = (0, \beta, -1), 
 \end{array}
 $$
 and $\beta = \frac{1 + \sqrt{5}}{2}$. 
 The optimal Grassmannian frame $\F_{e'}$
 has a frame bound 
$A=2$ (Sect. \ref{sII.1}). Consequently, the associated Parseval frame $\F_{e}=\{e_j\}_{j=1}^6$ 
 is obtained by rescaling 
$e_j:=\frac{1}{\sqrt{2}}e_j'$. 
The Gram matrix of $\mathcal{F}_e$ has the form of 
$$G_e = \frac{1}{2\sqrt{5}}   \begin{bmatrix}
    \sqrt{5} & 1 & 1 & -1 & 1 &1 \\
    1&\sqrt{5}&1&-1&-1&-1 \\
    1&1&\sqrt{5}&1&1&-1 \\
    -1&-1&1&\sqrt{5}&1&-1 \\
    1&-1&1&1&\sqrt{5}&1 \\
    1&-1&-1&-1&1&\sqrt{5}
\end{bmatrix}.$$
In the following, we present a result analogous to that obtained for Mercedes-type Parseval frames in Section~\ref{sIV.3}; its proof is deferred to the appendix.
\begin{thm}\label{ttt64}
    Let $\mathcal{F}_e$ denote the Parseval frame in $\mathbb{R}^3$ defined above, and let $H_{{\sf E}, e}$ 
 be the operator defined by \eqref{K5} in $\mathbb{R}^3$. Then eigenvalues of $H_{{\sf E}, e}$ satisfy the equation
 $$\sum_{\sigma \in S_3} \text{sgn}(\sigma)\prod_{i=1}^3 W_{i, \sigma(i)}(\lambda) = 0,$$
 where  
 $$W_{i,j}(\lambda) = \left\{ \begin{aligned}
        E_j - E_{3+i} \quad &\text{for} \quad i+j=4 \\
        \beta E_j - \bar{\beta}E_{3+i} -\frac{\lambda}{\alpha} \quad &\text{for} \quad i+j\neq 4
    \end{aligned} \right. $$
and $\bar{\beta} = \frac{1 - \sqrt{5}}{2}$.
\end{thm}

\section{Appendix}
\subsection{The proof of Theorem \ref{ttt63}}
Since 
$\F_e=\{e_j\}_{j\in\mathbb{J}}$
 is a Mercedes-type frame, its Gram matrix
 has the special form given in \eqref{gramm}. Denote by $\{\delta_i\}_{i=1}^k$  the canonical basis in $\mathbb{C}^k$.  An analysis of \eqref{gramm} shows that 
$G_e$	has eigenvectors 
$$
v_1,\dots v_{k-1}, \qquad v_i= \delta_1 - \delta_k, \qquad k=|\mathbb{J}|
$$ associated with eigenvalue $\lambda_1 = 1$ and one eigenvector $v_k = (1,1,\dots,1)^T$ associated with eigenvalue $\lambda_2 = 0$. Hence $G_e$ have the form $G_e = M^{-1}DM$, where the matrices $M, M^{-1}$ and $D$ admit the following block representation:
$$
M^{-1}=\frac{1}{k}\begin{bmatrix}
    \boldsymbol{1}_{1\times k-1} & 1 \\
    \boldsymbol{1}_{k-1} - kI_{k-1}& \boldsymbol{1}_{k-1\times 1}
\end{bmatrix}, \: 
D=\begin{bmatrix}
    I_{k-1} & \boldsymbol{0}_{k-1\times1} \\
    \boldsymbol{0}_{1\times k-1} & 0
\end{bmatrix},  \:
M=\begin{bmatrix}
    \boldsymbol{1}_{k-1\times 1} & -I_{k-1} \\
    1& \boldsymbol{1}_{1\times k-1}
\end{bmatrix},
$$
where $\boldsymbol{1}$ and $\boldsymbol{0}$ denote matrices whose dimensions are indicated by subscripts, with all entries of $\boldsymbol{1}$
 equal to $1$ and all entries of $\boldsymbol{0}$
 equal to $0$ (zero matrix).

By Corollary \ref{xoxo35}, the spectrum of $H_{{\sf E}, e}$ is completely determined by that  of the operator $B =G_e \mathcal{E} G_e$ defined in \eqref{xx2} and acting on $l_2(\mathbb{J})=\mathbb{C}^k$.
Using the decomposition of $G_e$ we get
$$
B =G_e \mathcal{E} G_e=M^{-1} D M \mathcal{E}M^{-1}DM=M^{-1}CM.$$
This means that $\sigma(B)=\sigma(C)$, where 
$$C =D M \mathcal{E}M^{-1}D = \frac{1}{k}\begin{bmatrix}
    E_1 - E_2 & \cdots & E_1-E_2 & 0 \\
    \vdots & \ddots & \vdots & \vdots \\
    E_1 - E_k & \cdots & E_1-E_k & 0 \\
    0 & 0& 0 & 0
\end{bmatrix} + \text{diag}\{E_2,E_3, \dots, E_k, 0\}.$$

Consider the determinant $\det (C-\lambda I)$. By reducing the last diagonal element by $-\lambda$, subtracting the first column from columns $2,3,\dots,k-1$, and introducing the notation $P_i = P_i(\lambda):= E_i - \lambda$, 
we obtain
\begin{equation}\label{bro1}
\det (C-\lambda I) =-\lambda\det\begin{bmatrix}
    \frac{P_1}{k} - \frac{P_2}{k} + P_2 & -P_2 & -P_2 &\cdots & -P_2 \\
    \frac{P_1}{k} - \frac{P_3}{k} & P_3 & 0 &\cdots & 0\\
    \frac{P_1}{k} - \frac{P_4}{k} & 0 & P_4 &\cdots & 0\\
    \vdots & \vdots & \vdots &\ddots & 0 \\
    \frac{P_1}{k} - \frac{P_k}{k}& 0 & 0 &\cdots & P_k
\end{bmatrix}.
\end{equation}
Therefore, the spectrum of 
$C$ contains 
$0$ as an eigenvalue, along with additional eigenvalues $\lambda_1, \lambda_2, \ldots, \lambda_{k-1}$ that remain to be determined and for which the determinant of the above matrix must vanish.
Applying the Laplace expansion along the first column of the matrix in \eqref{bro1}, we obtain
\begin{equation}
    P_2 M_{1,1}  + \sum_{i=1}^{k-1}(-1)^{i+1}\left( \frac{P_1}{k} - \frac{P_{i+1}}{k} \right)M_{i,1} = 0,
    \label{eigseq}
\end{equation}
where $M_{i,1}$ are the corresponding minors. A straightforward calculation shows that
$$
M_{i,1} = (-1)^{i-1}\frac{1}{P_{i+1}}\prod_{j=2}^kP_j.$$
Substituting that into \eqref{eigseq} we get
\begin{align*}
    0 &= P_2 M_{1,1}  + \sum_{i=1}^{k-1}(-1)^{i+1}\left( \frac{P_1}{k} - \frac{P_{i+1}}{k} \right)M_{i,1} = \prod_{j=2}^kP_j + \sum_{i=1}^{k-1}\left( \frac{P_1}{k} - \frac{P_{i+1}}{k} \right)\frac{1}{P_{i+1}}\prod_{j=2}^kP_j \\
    &= \prod_{j=2}^kP_j\left[1+\frac{1}{k}\sum_{i=1}^{k-1}\frac{P_1}{P_{i+1}}-\frac{k-1}{k}\right]  =  \frac{1}{k}\prod_{j=2}^kP_j\sum_{i=1}^k\frac{P_1}{P_i}=\frac{1}{k}\sum_{i=1}^{k}\prod_{j\neq i}^kP_j=\frac{1}{k}\sum_{i=1}^{k}\prod_{j\neq i}^k(E_j-\lambda).
\end{align*}
Applying Corollary \ref{xoxo35} and taking into account that ${\bf e}[\F_e]=1$
(see Remark \ref{uuu1}), we complete the proof.

\subsection{The proof of Theorem \ref{ttt64}}
\begin{proof}
Recall that $\beta = \frac{1 +\sqrt{5}}{2}$, $\bar{\beta} = \frac{1 -\sqrt{5}}{2}$ and $\alpha = \frac{1}{\sqrt{5}}$.  Given matrices 
$$A = \begin{bmatrix}
        -\beta & -\beta & -1 \\
        \beta & 1 & \beta \\
        -1 & -\beta & -\beta \\
    \end{bmatrix}, \qquad \qquad \bar{A} = \begin{bmatrix}
        -\bar{\beta} & -\bar{\beta} & -1 \\
        \bar{\beta} & 1 & \bar{\beta} \\
        -1 & -\bar{\beta} & -\bar{\beta} \\
    \end{bmatrix},$$
$$U = (A-\bar{A})^{-1} = \frac{1}{2}\begin{bmatrix}
    -\alpha & \alpha & \alpha \\
    -\alpha & -\alpha &-\alpha \\
    \alpha & \alpha & -\alpha
\end{bmatrix}.$$
We have the following representation of the Gram matrix $G_e$ using block matrices

\begin{equation}\label{new14} G_e = M^{-1}DM = \begin{bmatrix}
    U & -U \\
    -\bar{A}U & I+\bar{A}U
\end{bmatrix} \cdot \begin{bmatrix}
    I & 0 \\
    0 & 0
\end{bmatrix} \cdot \begin{bmatrix}
    A & I \\
    \bar{A} & I
\end{bmatrix}.
\end{equation}

By Corollary \ref{xoxo35}, the eigenvalues of $H_{{\sf E}, e}$ are   determined by that  of the operator $G_e \mathcal{E} G_e$ acting on $l_2(\mathbb{J})=\mathbb{C}^6$. Using \eqref{new14} we see that the eigenvalues of $G_e \mathcal{E} G_e$ are the same as of the matrix $C= D M \mathcal{E}M^{-1}D$. Splitting $\mathcal{E}$ into $\mathcal{E}_1 = \text{diag}\{E_1,E_2,E_3\}$ and $\mathcal{E}_2 = \text{diag}\{E_4,E_5,E_6\}$, we get
\begin{align*}
    C &= D M \mathcal{E}M^{-1}D = \\ 
    &=\begin{bmatrix}
    I & 0 \\
    0 & 0
\end{bmatrix}\cdot \begin{bmatrix}
    A & I \\
    \bar{A} & I
\end{bmatrix} \cdot \begin{bmatrix}
    \mathcal{E}_1 & 0 \\
    0 & \mathcal{E}_2
\end{bmatrix}\cdot \begin{bmatrix}
    U & -U \\
    -\bar{A}U & I+\bar{A}U
\end{bmatrix} \cdot \begin{bmatrix}
    I & 0 \\
    0 & 0
\end{bmatrix} = \\
&= \begin{bmatrix}
    (A\mathcal{E}_1 - \mathcal{E}_2\bar{A})U & 0 \\
    0 & 0
\end{bmatrix}.
\end{align*}
This means that the eigenvalues of $H_{{\sf E}, e}$ correspond to the eigenvalues of $(A\mathcal{E}_1 - \mathcal{E}_2\bar{A})U$.  
Calculating $\det{((A\mathcal{E}_1 - \mathcal{E}_2\bar{A})U - \lambda I)}$, we get
\begin{align*}
    \det{((A\mathcal{E}_1 - \mathcal{E}_2\bar{A})U - \lambda I)} &= \det{(A\mathcal{E}_1 - \mathcal{E}_2\bar{A} - \lambda U^{-1})U} = \\ 
    &= \det{(A\mathcal{E}_1 - \mathcal{E}_2\bar{A} - \lambda (A - \bar{A}))} \det U,
\end{align*}
where $\det U \neq 0$ and
$$A\mathcal{E}_1 - \mathcal{E}_2\bar{A} - \lambda (A - \bar{A}) = \begin{bmatrix}
    -\beta E_1 + \bar{\beta}E_4 & -\beta E_2 + \bar{\beta} E_4 & -E_3+E_4 \\
        \beta E_1 - \bar{\beta}E_5& E_2 - E_5 & \beta E_3 - \bar{\beta}E_5 \\
        -E_1 +E_6 & -\beta E_2 + \bar{\beta} E_6 & -\beta E_3 +\bar{\beta}E_6 \\
\end{bmatrix} - \frac{\lambda}{\alpha} \begin{bmatrix}
     -1 & -1 & 0\\
     1& 0& 1\\
     0& -1& -1\\
\end{bmatrix}.$$
By changing the signs of the first and last rows and using the definition of the determinant, we obtain
$$\det{(A\mathcal{E}_1 + \mathcal{E}_2\bar{A} - \lambda (A - \bar{A}))} = \sum_{\sigma \in S_3} \text{sgn}(\sigma)\prod_{i=1}^3 W_{i, \sigma(i)}(\lambda) = 0,$$
where
 $$W_{i,j}(\lambda) = \left\{ \begin{aligned}
        E_j - E_{3+i} \quad &\text{for} \; i+j=4 \\
        \beta E_j - \bar{\beta}E_{3+i} -\frac{\lambda}{\alpha} \quad &\text{for} \; i+j\neq 4
    \end{aligned} \right. $$
\end{proof}

\section*{Acknowledgments} { This work was partially supported by the AGH UST Faculty of Applied Mathematics statutory tasks within the subvention of 
the Ministry of Science and Higher Education of Poland.
SK gratefully acknowledges Dr. Roberto Beneduci for his valuable and constructive discussions on the theory of POVMs during his visit to AGH University.}

\end{document}